\documentclass[reqno]{amsart}

\usepackage{amssymb,amsxtra,latexsym}

\usepackage[sc]{mathpazo}
\usepackage{mathrsfs}
\usepackage[T1]{fontenc}
\usepackage{textcomp}

\usepackage[cmtip,matrix,arrow,curve]{xy}
\CompileMatrices
\newdir{((}{{\lhook\kern-.1em}}
\newdir{))}{{\rhook\kern-.1em}}
\newdir{ >}{{}*!/-5pt/@{>}}

\usepackage{booktabs}

\newcommand\note[1]%
{$^\dagger$\marginpar{\footnotesize{$^\dagger${#1}}}}

\divide\count253 by 60 
\divide\count255 by 60
\multiply\count255 by 60 
\advance\count254 by -\count255 

\def\today{\number\year-\ifnum\month<10
0\fi\number\month-\ifnum\day<10 0\fi\number\day}

\def\hour{\ifnum\count253<10
0\number\count253\else\number\count253\fi}

\def\minute{\ifnum\count254<10
0\number\count254\else\number\count254\fi}

\swapnumbers
\newtheorem{theorem}{Theorem}[section]
\newtheorem{lemma}[theorem]{Lemma}
\newtheorem{proposition}[theorem]{Proposition}
\newtheorem{corollary}[theorem]{Corollary}

\theoremstyle{definition}

\newtheorem{example}[theorem]{Example}

\newtheorem{remark}[theorem]{Remark}

\numberwithin{equation}{section}

\newcommand\lie{\mathfrak}

\newcommand{\g}{\lie{g}}

\renewcommand{\t}{\lie{t}}

\newcommand\bb[1]{{\text{\bf#1}}}

\newcommand\N{\bb{N}}
\newcommand\Z{\bb{Z}} 
\newcommand\Q{\bb{Q}}
\newcommand\R{\bb{R}} 
\newcommand\C{\bb{C}}
\renewcommand\k{\bb{k}}

\renewcommand\P{\bb{P}}

\newcommand\ca{\mathscr}

\newcommand\X{\ca{X}}

\newcommand\func[1]{\operatorname{\mathrm{#1}}}

\renewcommand\det{\func{det}}
\newcommand\End{\func{End}}

\newcommand\Hom{\func{Hom}}

\renewcommand\index{\func{index}}

\newcommand\pr{\func{pr}}

\newcommand\Rep{\operatorname{\lie{Rep}}}

\newcommand\sheafhom{\operatorname{\ca{H}\!\mathit{om}}}

\newcommand\Spec{\func{Spec}}

\newcommand\group[1]{{\text{\bf#1}}}

\newcommand\SU{\group{SU}}
\newcommand\U{\group{U}}

\newcommand\abs[1]{\lvert#1\rvert}

\newcommand\quot[1][\kern.3ex]{/\kern-.7ex/_{\kern-.4ex#1}}
\newcommand\bigquot[1][\,\,]{\big/\kern-.85ex\big/_{\!\!#1}}

\newcommand\powl{[\kern-.3ex[}
\newcommand\powr{]\kern-.3ex]}
\newcommand\bigpowl{\bigl[\kern-.6ex\bigl[}
\newcommand\bigpowr{\bigr]\kern-.6ex\bigr]}

\newcommand\sur{\mathrel{\to\kern-1.8ex\to}}
\newcommand\iso{\mathrel{\hookrightarrow\kern-1.8ex\to}}

\newcommand\longto{\longrightarrow}

\newcommand\longsur{\mathrel{\longrightarrow\kern-1.8ex\to}}

\newcommand\Edh{{D\mkern-15mu\text{\raise0.1em\hbox{--}}\mkern8mu}}

\renewcommand\d{\group{d}}

\newcommand\zerodots%
{\smash{\makebox[0pt]{$_{\lower8pt\hbox{\Large$0$}}%
\ddots^{\lower3pt\hbox{\Large$0$}}$}}}
\newcommand\bigzerodots%
{\smash{\makebox[0pt]{$_{\lower6pt\hbox{\Large$0$}}%
\ddots^{\hbox{\Large$0$}}$}}}

\newcommand\pt{{\rm pt}}

\renewcommand\subset{\subseteq}
\renewcommand\supset{\supseteq}

\begin{document}


\title[Divided differences and character formula in K-theory]{Divided
differences and the Weyl character formula in equivariant K-theory}

\author{Megumi Harada}

\email{Megumi.Harada@math.mcmaster.ca}

\address{Department of Mathematics and Statistics, McMaster
University, Hamilton, ON L8S 4K1, Canada}

\author{Gregory D. Landweber}

\email{gregland@bard.edu}

\address{Department of Mathematics, Bard College, Annandale-on-Hudson,
NY 12504-5000, USA}

\author{Reyer Sjamaar}

\email{sjamaar@math.cornell.edu}

\address{Department of Mathematics, Cornell University, Ithaca, NY
14853-4201, USA}


\date{4 March 2010}


\begin{abstract}
Let $X$ be a topological space and $G$ a compact connected Lie group
acting on $X$.  Atiyah proved that the $G$-equivariant K-group of $X$
is a direct summand of the $T$-equivariant K-group of $X$, where $T$
is a maximal torus of $G$.  We show that this direct summand is equal
to the subgroup of $K_T^*(X)$ annihilated by certain divided
difference operators.  If $X$ consists of a single point, this
assertion amounts to the Weyl character formula.  We also give
sufficient conditions on $X$ for $K_G^*(X)$ to be isomorphic to the
subgroup of Weyl invariants of $K_T^*(X)$.
\end{abstract}


\maketitle

\tableofcontents


\section*{Introduction}\label{section;introduction}

Hermann Weyl's theorem of the maximal torus allows one for many
purposes to pass from a compact connected Lie group $G$ to a maximal
torus $T$.  A famous example of this principle is the Weyl character
formula, which enables one to compute the characters of the
irreducible representations of $G$ in terms of the characters of $T$
and the action of the Weyl group on $T$.

One of the results of this paper is an extension of the Weyl character
formula to the equivariant K-theory of a compact $G$-space $X$.
Atiyah \cite{atiyah;bott-periodicity-index} proved that the
restriction map from the $G$-equivariant K-ring $K_G^*(X)$ to the
$T$-equivariant K-ring $K_T^*(X)$ has a natural left inverse.  This
``wrong-way'' or pushforward homomorphism is defined by means of the
Dolbeault operator associated with an invariant complex structure on
the homogeneous space $G/T$.  Although Atiyah proved that $K_G^*(X)$
is a direct summand of $K_T^*(X)$, he did not tell whether this direct
summand is determined by the Weyl group action on $K_T^*(X)$.  In
fact, it is easy to see that $K_G^*(X)$ is contained in the Weyl
invariants of $K_T^*(X)$, but an example of McLeod
\cite{mcleod;kunneth-equivariant} (see also Example
\ref{example;mcleod}) shows that this inclusion is in general not an
equality.

In this paper we show that the action of the Weyl group $W$ on
$K_T^*(X)$ extends to an action of a Hecke ring $\ca{D}$ generated by
divided difference operators, which was introduced in the context of
Schubert calculus by Demazure
\cite{demazure;desingularisation,demazure;nouvelle-formule}.  This
$\ca{D}$-action is analogous to a $\ca{D}$-action on the K-theory of
the classifying space $BT$ previously defined by Bressler and Evens
\cite{bressler-evens;schubert-braid-generalized}.  The ring $\ca{D}$
contains a left ideal $I(\ca{D})$, which we dub the \emph{augmentation
left ideal} by analogy with the augmentation ideal of the group ring
of $W$.  Our main result, Theorem \ref{theorem;d}, is that $K_G^*(X)$
is isomorphic to the subring of $K_T^*(X)$ annihilated by $I(\ca{D})$.
In other words, a $T$-equivariant class is $G$-equivariant if and only
if it is killed by the divided difference operators.  This result is
an application of the work of Demazure and of various structural
theorems regarding the rings $R(T)$ and $\ca{D}$ due to Pittie
\cite{pittie;homogeneous-vector}, Steinberg \cite{steinberg;pittie},
Kazhdan and Lusztig \cite{kazhdan-lusztig;deligne-langlands-hecke},
and Kostant and Kumar
\cite{kostant-kumar;equivariant-k-theory-flag;journal}.

It turns out (Theorem \ref{theorem;weyl}) that the divided difference
operator associated with the longest Weyl group element acts on
$K_T^*(X)$ as a projection onto the direct summand $K_G^*(X)$.  For a
space $X$ consisting of a single point, this statement is equivalent
to the Weyl character formula.  Thus our results can be viewed as
``lifting'' the Weyl character formula from a point to an arbitrary
$G$-space.

There is a natural condition under which $K_G^*(X)$ is isomorphic to
the Weyl invariant part of $K_T^*(X)$, namely that the Weyl
denominator, viewed as an element of the representation ring $R(T)$,
should not be a zero divisor in $K_T^*(X)$ (Theorem
\ref{theorem;abelian}).  This condition is satisfied for interesting
classes of spaces, such as nonsingular projective varieties on which
$G$ acts linearly (Corollary \ref{corollary;abelian-torsion}).

A by-product of our work is a form of duality.  There is a natural
bilinear pairing on the $K_G^*(X)$-module $K_T^*(X)$ defined by
pushing forward to $K_G^*(X)$ the product of two classes in
$K_T^*(X)$.  Proposition \ref{proposition;orthogonal} states that
$K_T^*(X)$ is isomorphic to its dual via this ``intersection''
pairing.

All these results generalize when we replace the maximal torus $T$ by
a closed connected subgroup of $G$ which contains $T$.  In this manner
we obtain a lift of the Gross-Kostant-Ramond-Sternberg character
formula to a topological $G$-space.  We hope to take this up in a
sequel to this paper.

The treatment of these problems in equivariant K-theory is parallel to
the case of Borel's equivariant cohomology theory $H_G^*(X)$, which
was covered in \cite{holm-sjamaar;torsion-abelianization}.  (For
instance, for reasonable coefficient rings we have $H_G^*(X)\cong
H_T^*(X)^W$ provided that the discriminant of $G$ is not a zero
divisor in $H_T^*(X)$.  Thus the Weyl denominator plays a r\^ole in
K-theory analogous to that of the discriminant in cohomology.)
However, there are several significant differences between the two
theories as well, and for this reason we have chosen to provide full
details.  One such difference is the fact that the restriction map
$H_G^*(X)\to H_T^*(X)$ becomes injective only after inverting the
torsion primes of $G$.  In K-theory these primes cause no trouble.

Our results have counterparts in the equivariant K-theory of algebraic
varieties, as follows from work of Merkurjev
\cite{merkurjev;comparison-equivariant-ordinary}, and in the
equivariant K-theory and KK-theory of nuclear $C^*$-algebras, as
follows from work of Rosenberg and Schochet
\cite{rosenberg-schochet;equivariant-k-theory}.  We present the case
of algebraic K-theory in Section \ref{section;scheme}.

We thank Samuel Evens and the referee for useful suggestions.

\section{Divided differences}\label{section;operator}

\subsection*{Notation}

Throughout this paper, with the exception of
Section~\ref{section;scheme}, $G$ denotes a compact connected Lie
group, $\X(G)=\Hom(G,\U(1))$ the character group of $G$, and $R(G)$
the Grothendieck ring of finite-dimensional complex $G$-modules.  We
choose once and for all a maximal torus $T$ of $G$.  Let $j\colon T\to
G$ be the inclusion map and $W=N_G(T)/T$ the Weyl group.  Then $R(T)$
is canonically isomorphic to the group ring $\Z[\X(T)]$ and the
restriction homomorphism $j^*\colon R(G)\to R(T)$ induces an
isomorphism $R(G)\cong R(T)^W$.  (See e.g.\
\cite[\S\,IX.3]{bourbaki;groupes-algebres}.)  Let $\ca{R}\subset\X(T)$
be the root system of $(G,T)$.  As in
\cite[\S\,VI.3]{bourbaki;groupes-algebres}, we denote by $e^\lambda$
the element of $R(T)$ defined by a character $\lambda\in\X(T)$.  We
fix a basis of the root system $\ca{R}$ and we let
$$\rho=\frac12\sum_{\alpha\in\ca{R}^+}\alpha\in\frac12\X(T)$$
be the half-sum of the positive roots.  

\subsection*{Holomorphic induction}

See \cite[\S\,4]{atiyah;bott-periodicity-index},
\cite[\S\,II.5]{atiyah-bott;lefschetz-fixed-elliptic-complex},
\cite{bott;homogeneous-differential} and
\cite{segal;equivariant-k-theory;;1968} for the material in this
subsection.  Let $G_\C$ be the complexification of $G$ and let $B^-$
be the Borel subgroup of $G_\C$ with Lie algebra
$$\lie{b}^-=\t_\C\oplus\bigoplus_{\alpha\in\ca{R}^+}\g_\C^{-\alpha},$$
where $\g_\C^\alpha$ denotes the root space of the root $\alpha$.  The
inclusion $G\to G_\C$ descends to a $G$-equivariant diffeomorphism
$G/T\cong G_\C/B^-$.

Let $V$ be a finite-dimensional complex $T$-module.  Extend $V$ to a
$B^-$-module by letting the nilradical of $B^-$ act trivially.  Let
$(\Lambda^{0,*}E_V,\bar{\partial}_V)$ be the Dolbeault complex on
$G_\C/B^-$ with coefficients in the homogeneous holomorphic vector
bundle $E_V=G_\C\times^{B^-}V$.  The equivariant index of
$\bar{\partial}_V$ is a virtual $G$-module and depends only on the
class of $V$ in $R(T)$.  The map $j_*\colon R(T)\to R(G)$ defined by
$[V]\mapsto\index(\bar{\partial}_V)$ is the \emph{pushforward
homomorphism} or \emph{holomorphic induction map}.  It satisfies
$j_*(1)=1$ and is $R(G)$-linear,
\begin{equation}\label{equation;linear}
j_*(j^*(v)u)=vj_*(u)
\end{equation}
for all $u\in R(T)$ and $v\in R(G)$.  It follows that $j_*j^*(v)=v$
for all $v\in R(G)$.  The Lefschetz fixed point formula gives an
expression for the $R(G)$-linear endomorphism $j^*j_*$ of $R(T)$,
namely
\begin{equation}\label{equation;push-pull-weyl}
j^*j_*(u)=\frac{J(u)}{\d}
\end{equation}
for all $u\in R(T)$.  Here $J$ is the $\Z$-linear endomorphism of
$R(T)$ defined by
\begin{equation}\label{equation;anti-invariant}
J(u)=\sum_{w\in W}\det(w)w\odot u,
\end{equation}
where $w\odot u=e^{-\rho}w(e^\rho u)$ is the $\rho$-shifted $W$-action
on $R(T)$, which is well-defined because $\rho-w(\rho)\in\X(T)$ for
all $w\in W$.  The denominator is
\begin{equation}\label{equation;denominator}
\d=J(1)=\sum_{w\in
W}\det(w)w\odot1=\prod_{\alpha\in\ca{R}^+}(1-e^{-\alpha}).
\end{equation}
By the Borel-Weil theorem, $j_*(e^\lambda)=[M_\lambda]$ for all
dominant $\lambda\in\X(T)$, where $M_\lambda$ is the irreducible
$G$-module with highest weight $\lambda$.  Therefore
\eqref{equation;push-pull-weyl} is equivalent to the Weyl character
formula.

\subsection*{Demazure's operators}

Demazure \cite{demazure;desingularisation,demazure;nouvelle-formule}
noticed that the endomorphism $j^*j_*$ of $R(T)$ can be factorized
into a composition of $N$ operators, where $N=\abs{\ca{R}^+}$.  These
operators are defined as follows.

Let $\alpha$ be a root and let $s_\alpha\in W$ be the reflection in
$\alpha$.  For every $\lambda\in\X(T)$, the element
$e^\lambda-e^{-\alpha}e^{s_\alpha(\lambda)}$ of $R(T)$ is uniquely
divisible by $1-e^{-\alpha}$, so we have a $\Z$-linear endomorphism
$\delta_\alpha$ of $R(T)$ defined by
\begin{equation}\label{equation;operator}
\delta_\alpha(u)=\frac{u-e^{-\alpha}s_\alpha(u)}{1-e^{-\alpha}}.
\end{equation}
This map, known in the recent literature as an \emph{isobaric divided
difference operator}, was introduced by Demazure in
\cite{demazure;nouvelle-formule} and denoted by $\Lambda^0_\alpha$
there.  He considered only the simply connected case, but
$\delta_\alpha$ is well-defined for any $G$.  The following properties
are easily deduced from the definition:
\begin{gather}
\label{equation;square}
\delta_\alpha(1)=1,\qquad\delta_\alpha^2=\delta_\alpha,
\\
\label{equation;reflect}
s_\alpha\delta_\alpha=\delta_\alpha,\qquad\delta_\alpha
s_\alpha=\delta_{-\alpha}=1+e^\alpha-e^\alpha\delta_\alpha,\qquad
w\delta_\alpha w^{-1}=\delta_{w(\alpha)},
\\
\label{equation;leibniz}
\delta_\alpha(u_1u_2)=\delta_\alpha(u_1)u_2
+s_\alpha(u_1)(\delta_\alpha(u_2)-u_2),
\end{gather}
for all $\alpha\in\ca{R}$, $w\in W$ and $u_1$, $u_2\in R(T)$.  In
addition, Demazure defined the operators (cf.\ also his earlier papers
\cite{demazure;invariants-symetriques-entiers,%
demazure;desingularisation})
\begin{equation}\label{equation;operator'}
\delta'_\alpha(u)=\frac{u-s_\alpha(u)}{1-e^{-\alpha}}.
\end{equation}
It is plain that
\begin{equation}\label{equation;two-operators}
\delta'_\alpha(1)=0,\qquad(\delta'_\alpha)^2=\delta'_\alpha,
\qquad\delta'_\alpha=e^\alpha(\delta_\alpha-1),
\end{equation}
where in the first equality $1$ denotes the identity element of $R(T)$
and in the third the identity automorphism of $R(T)$.  The product
rule is
\begin{gather}
\label{equation;product'}
\delta'_\alpha(u_1u_2)=\delta'_\alpha(u_1)u_2
+s_\alpha(u_1)\delta'_\alpha(u_2),\\
\label{equation;product}
\delta_\alpha(u_1u_2)=\delta_\alpha(u_1)u_2
+e^{-\alpha}s_\alpha(u_1)\delta'_\alpha(u_2),
\end{gather}
where the second line is just an alternative form of
\eqref{equation;leibniz}.

Demazure proved that the $\delta_\alpha$ satisfy braid relations and
deduced the following statement.  (His proof has an error, which was
corrected by Andersen \cite{andersen;schubert-demazure} and Ramanan
and Ramanathan \cite{ramanan-ramanathan;normality-flag}.)

\begin{theorem}[{\cite[Th\'eor\`eme~1]{demazure;nouvelle-formule}}]
\label{theorem;operators}
For every $w\in W$ and for every reduced expression
$w=s_{\beta_1}s_{\beta_2}\cdots s_{\beta_l}$ in terms of simple
reflections, the composition
$\delta_{\beta_1}\delta_{\beta_2}\cdots\delta_{\beta_l}$ takes the
same value $\partial_w$; and the composition
$\delta'_{\beta_1}\delta'_{\beta_2}\cdots\delta'_{\beta_l}$ takes the
same value $\partial'_w=e^\rho\partial_we^{-\rho}$.
\end{theorem}

(A comment on the last identity: if $\rho\not\in\X(T)$, then
multiplication by $e^{\pm\rho}$ does not preserve $\Z[\X(T)]$, but the
operator $e^\rho\partial_we^{-\rho}$ does preserve $\Z[\X(T)]$ and is
equal to $\partial'_w$.)

It follows from \eqref{equation;product'} and \eqref{equation;product}
that the endomorphisms $\partial_w$ and $\partial'_w$ of $R(T)$ are
$R(T)^W$-linear.  Let $w_0$ be the longest Weyl group element.  The
most important property of Demazure's operators is the following
formula, which relates the ``top'' operator $\partial_{w_0}$ to the
Weyl character formula,
$$\partial_{w_0}(u)=J(u)/\d$$
for all $u\in R(T)$ (\cite[Proposition~3]{demazure;nouvelle-formule}).
Because of \eqref{equation;push-pull-weyl} this is equivalent to
\begin{equation}\label{equation;long-symmetric}
\partial_{w_0}=j^*j_*.
\end{equation}
Demazure has given a similar characterization of the operators
$\partial_w$ for all $w\in W$; cf.\ Theorem \ref{theorem;demazure}.

\section{The Hecke algebra $\ca{D}$}\label{section;algebra}
 
This section is in part a review of the work of Kazhdan and Lusztig
\cite{kazhdan-lusztig;deligne-langlands-hecke} and Kostant and Kumar
\cite{kostant-kumar;equivariant-k-theory-flag;journal} on the
representation ring $R(G)$ and the associated Hecke algebra $\ca{D}$.
All results stated here follow readily from their work.  As far as we
are aware, the main novelty is the introduction of the augmentation
left ideal of $\ca{D}$, which will play an important role in Section
\ref{section;push-pull}.  The notation is as in Section
\ref{section;operator}.

\subsection*{The augmentation left ideal}

Let $\ca{E}=\End_{R(G)}(R(T))$ be the $R(G)$-algebra of $R(G)$-linear
endomorphisms of $R(T)$.  Let $\ca{D}$ be the subalgebra of $\ca{E}$
generated by the $\delta_\alpha$ and the elements of $R(T)$ (regarded
as multiplication operators).  It follows from
\eqref{equation;operator} that $\ca{D}$ contains the group ring
$\Z[W]$ (viewed as an algebra of endomorphisms of $R(T)$).  It follows
from Theorem \ref{theorem;operators} that $\partial_w$,
$\partial'_w\in\ca{D}$ for all $w$.  We define the \emph{augmentation
left ideal} of $\ca{D}$ to be the annihilator of the identity element
$1\in R(T)$,
$$
I(\ca{D})=\{\,\Delta\in\ca{D}\mid\Delta(1)=0\,\}.
$$
Note that $I(\ca{D})$ contains the augmentation ideal $I(W)$ of
$\Z[W]$ (whence its name), as well as the operators $\delta'_\alpha$
for all $\alpha\in\ca{R}$ and $\partial'_w$ for all $w\ne1$.

The next result follows immediately from work of Kostant and Kumar.
It says that $\ca{D}$ and $I(\ca{D})$ are free left $R(T)$-modules,
and that as a ring $\ca{D}$ is isomorphic to the Hecke algebra (cf.\
\cite[\S\,2.12]{kazhdan-lusztig;deligne-langlands-hecke}) over $\Z$ of
the extended affine Weyl group $\X(T)\rtimes W$.

\begin{theorem}\label{theorem;kostant-kumar}
\begin{enumerate}
\item\label{item;kostant-kumar}
$(\partial_w)_{w\in W}$ is a basis of the left $R(T)$-module $\ca{D}$.
\item\label{item;kostant-kumar'}
$(\partial'_w)_{w\in W}$ is a basis of the left $R(T)$-module
$\ca{D}$.
\item\label{item;kostant-kumar-ideal}
$(\partial'_w)_{w\ne1}$ is a basis of the left $R(T)$-module
$I(\ca{D})$.
\item\label{item;hecke}
Let $S$ be the set of simple reflections.  The multiplication law of
$\ca{D}$ is determined by that of $R(T)$ and by the rules
$$
[\partial'_s,u]=\partial'_s(u)s,\qquad\partial'_s\partial'_w=
\begin{cases}
\partial'_{sw}&\text{if $l(sw)=l(w)+1$}\\
\partial'_w&\text{if $l(sw)=l(w)-1$}
\end{cases}
$$
for all $s\in S$, $w\in W$ and $u\in R(T)$.
\end{enumerate}
\end{theorem}

\begin{proof}
Let $K$ be the fraction field of $R(T)$.  Then $K^W$ is the fraction
field of $R(T)^W\cong R(G)$.  Let $\ca{A}$ be the subalgebra of the
twisted group algebra $K[W]$ which stabilizes $R(T)$.  Then
$\delta_\alpha\in\ca{A}$ for all $\alpha\in\ca{R}$, so $\ca{D}$ is a
subalgebra of $\ca{A}$.  By
\cite[Theorem~2.9]{kostant-kumar;equivariant-k-theory-flag;journal}
(where $K[W]$ is denoted by $Q_W$ and $\ca{A}$ by $Y$), $\ca{A}$ is a
free left $R(T)$-module with basis $(\partial_w)_{w\in W}$.  This
shows that $\ca{A}=\ca{D}$, and hence \eqref{item;kostant-kumar}.  It
follows from the Leibniz rule \eqref{equation;leibniz} and from
\eqref{equation;two-operators} that the operators $\partial'_w$
generate the same left $R(T)$-submodule of $\ca{E}$ as the operators
$\partial_w$.  Therefore $(\partial'_w)_{w\in W}$ is also a basis of
$\ca{D}$.  By virtue of the fact that $\partial'_1$ is the identity
operator and that $\partial'_w(1)=0$ for $w\ne1$, this implies that
$I(\ca{D})$ is the submodule spanned by the tuple
$(\partial'_w)_{w\ne1}$.  The commutation rule
$[\partial'_s,u]=\partial'_s(u)s$ is equivalent to the product rule
\eqref{equation;product'}.  That
$\partial'_s\partial'_w=\partial'_{sw}$ if $l(sw)=l(w)+1$ follows from
the definition of the $\partial'_w$.  If $l(sw)=l(w)-1$, there is a
reduced expression $w=ss_1\cdots s_k$ for $w$, so
$$
\partial'_s\partial'_w=\partial'_s\partial'_s\partial'_{s_1}\cdots
\partial'_{s_k}=\partial'_s\partial'_{s_1}\cdots\partial'_{s_k}
=\partial'_w
$$
by Theorem \ref{theorem;operators} and \eqref{equation;two-operators}.
Since $R(T)$ and the $\partial'_s$ generate $\ca{D}$, these rules
determine the algebra structure of $\ca{D}$.
\end{proof}

(It follows from
\cite{kostant-kumar;equivariant-k-theory-flag;journal} that
\eqref{item;kostant-kumar} and \eqref{item;kostant-kumar'} are true
also for the \emph{right} $R(T)$-module structure on $\ca{D}$.  The
left ideal $I(\ca{D})$, on the other hand, is not a right
$R(T)$-module.)

\subsection*{The $R(G)$-bilinear form}

Let $\ca{D}_0$ be the subalgebra of $\ca{E}=\End_{R(G)}(R(T))$
generated by the endomorphism $\partial_{w_0}$ and by the elements of
$R(T)$.  Then we have inclusions
$$\ca{D}_0\subset\ca{D}\subset\ca{E}.$$
Now assume that $\pi_1(G)$ is \emph{torsion-free}, i.e.\ that the
derived subgroup of $G$ is simply connected.  We claim that
$\ca{D}_0=\ca{D}=\ca{E}$ in that case.  The proof is based on the
symmetric pairing $\ca{P}\colon R(T)\times R(T)\to R(G)$ defined by
\begin{equation}\label{equation;pairing}
\ca{P}(u_1,u_2)=j_*(u_1u_2)
\end{equation}
for $a_1$, $a_2\in R(T)$, where $j_*$ is the holomorphic induction map
of Section \ref{section;operator}.  It follows from
\eqref{equation;linear} that this pairing is $R(G)$-bilinear.  Pittie
\cite{pittie;homogeneous-vector} and Steinberg \cite{steinberg;pittie}
showed that the $R(G)$-module $R(T)$ is free of rank $\abs{W}$.
Kazhdan and Lusztig
\cite[Proposition~1.6]{kazhdan-lusztig;deligne-langlands-hecke} showed
that the pairing $\ca{P}$ is nonsingular in the sense that the induced
map
$$\ca{P}^\sharp\colon R(T)\to R(T)^\vee$$
is an isomorphism, where $R(T)^\vee=\Hom_{R(G)}(R(T),R(G))$ denotes
the $R(G)$-dual of the module $R(T)$.  (See also Panin
\cite[Theorem~8.1]{panin;algebraic-k-theory-twisted-flag-varieties},
who noted that this fact is implicit in an older result of Hulsurkar
\cite{hulsurkar;verma-conjecture-weyl}; cf.\ also Merkurjev
\cite[Proposition~2.17]{merkurjev;comparison-equivariant-ordinary}.)
We regard $R(T)^\vee$ as an $R(T)$-module with scalar multiplication
given by $(u_1\cdot\phi)(u_2)=\phi(u_1u_2)$ for $u_1$, $u_2\in R(T)$
and $\phi\in R(T)^\vee$.  Similarly, we view $\ca{E}$ as an
$R(T)\otimes_{R(G)}R(T)$-module with scalar multiplication given by
$$(u_1\otimes u_2)\cdot\Delta(u)=u_1\Delta(u_2u)$$
for $u_1$, $u_2$, $u\in R(T)$ and $\Delta\in\ca{E}$.

\begin{proposition}\label{proposition;endomorphism}
Assume that $\pi_1(G)$ is torsion-free.
\begin{enumerate}
\item\label{item;dual}
$R(T)^\vee$ is a free $R(T)$-module of rank $1$ generated by the
pushforward map $j_*$.
\item\label{item;endomorphism}
As an $R(G)$-algebra, $\ca{E}$ is isomorphic to the matrix algebra
$\ca{M}_{\abs{W}}(R(G))$.  As an $R(T)\otimes_{R(G)}R(T)$-module,
$\ca{E}$ is free of rank $1$ on the generator $\partial_{w_0}$.
\item\label{item;expansion}
Let $(u_w)_{w\in W}$ be a basis of the $R(G)$-module $R(T)$.  For
every $\Delta\in\ca{E}$ there exist unique $b_{w,w'}\in R(G)$ such
that
$$\Delta(u)=\sum_{w,w'\in W}j^*(b_{w,w'})u_w\partial_{w_0}(u_{w'}u)$$
for all $u\in R(T)$.
\item\label{item;algebras}
We have $\ca{D}_0=\ca{D}=\ca{E}$.
\end{enumerate}
\end{proposition}

\begin{proof}
Observe first that $j_*\in R(T)^\vee$ because $j_*$ is $R(G)$-linear.
\eqref{item;dual} now follows immediately from the fact that
$\ca{P}^\sharp(1)=j_*$ and that $\ca{P}$ is nonsingular.  The first
observation in \eqref{item;endomorphism} is obvious from the fact that
$R(T)$ is free of rank $\abs{W}$.  Let $f$ be the composition of the
natural $R(T)\otimes_{R(G)}R(T)$-linear maps
$$
\xymatrix{
R(T)\otimes_{R(G)}R(T)\ar[r]^-{\ca{P}^\sharp\otimes1}&
R(T)^\vee\otimes_{R(G)}R(T)\ar[r]^-g&{\ca{E}},
}
$$
where $g$ is defined by $g(\phi\otimes u_1)(u_2)=j^*\phi(u_1u_2)$.
Then $g$ is an isomorphism because $R(T)$ is free and
$\ca{P}^\sharp\otimes1$ is an isomorphism because $\ca{P}$ is
nonsingular.  Moreover, $f(1\otimes1)=j^*j_*$, which is equal to
$\partial_{w_0}$ by \eqref{equation;long-symmetric}.  This proves the
second statement in \eqref{item;endomorphism}.  The $\abs{W}^2$-tuple
$(u_w\otimes u_{w'})_{w,w'\in W}$ is an $R(G)$-basis of
$R(T)\otimes_{R(G)}R(T)$, so \eqref{item;expansion} is a restatement
of \eqref{item;endomorphism}.  It follows from \eqref{item;expansion}
that $\ca{E}\subset\ca{D}_0$, which proves~\eqref{item;algebras}.
\end{proof}

For each character $\lambda\in\X(T)$, define $j_\lambda=e^\lambda\cdot
j_*\in R(T)^\vee$, i.e.\ $j_\lambda(u)=j_*(e^\lambda u)$ for all $u\in
R(T)$.  We call $j_\lambda$ the \emph{twisted} induction map with
coefficients in the one-dimensional $T$-module defined by $\lambda$.

\begin{corollary}\label{corollary;induct}
Assume that $\pi_1(G)$ is torsion-free.  For each $\lambda\in\X(T)$,
the twisted induction map $j_\lambda$ is a generator of the
$R(T)$-module $R(T)^\vee$.  The collection
$(j_\lambda)_{\lambda\in\X(T)}$ is a basis of the abelian group
$R(T)^\vee$.
\end{corollary}

\begin{proof}
The first assertion follows from Proposition
\ref{proposition;endomorphism}\eqref{item;dual} and the fact that
$e^\lambda$ is a unit in $R(T)$.  The second assertion follows from
Proposition \ref{proposition;endomorphism}\eqref{item;dual} and the
fact that the $e^\lambda$ form a basis of $R(T)$.
\end{proof}

\subsection*{Behaviour under covering maps}

Let $\tilde{G}$ be a second compact connected Lie group and let
$\phi\colon\tilde{G}\to G$ be a covering homomorphism.  Let
$\tilde{T}$ be the maximal torus $\phi^{-1}(T)$ of $\tilde{G}$.  The
pullback map $\phi^*\colon\ca{X}(T)\to\ca{X}(\tilde{T})$ is injective.
It induces an injective homomorphism
$$\phi^*\colon R(T)\to R(\tilde{T}),$$
and it maps the root system of $G$ bijectively to that of $\tilde{G}$.
We shall identify the two root systems via this bijection.  For every
root $\alpha$, the endomorphism $\delta_\alpha$ extends to an
$R(\tilde{G})$-linear endomorphism $\tilde{\delta}_\alpha$ of
$R(\tilde{T})$ given by the same formula as \eqref{equation;operator}.
Similarly, $\delta'_\alpha$ extends to an $R(\tilde{G})$-linear
endomorphism $\tilde{\delta}'_\alpha$.  Also, multiplication by an
element $u\in R(T)$ extends in an obvious way to a multiplication
operator on $R(\tilde{T})$.  Thus we have defined an algebra
homomorphism $\bar{\phi}\colon\ca{D}\to\tilde{\ca{D}}$, where
$\tilde{\ca{D}}$ is the algebra of endomorphisms of $R(\tilde{T})$
generated by the $\tilde{\delta}_\alpha$ and by $R(\tilde{T})$.
Observe that $\bar{\phi}$ maps $I(\ca{D})$ to $I(\tilde{\ca{D}})$.

\begin{lemma}\label{lemma;covering}
The homomorphism $\bar{\phi}\colon\ca{D}\to\tilde{\ca{D}}$ is
injective and induces isomorphisms of left $R(\tilde{T})$-modules
$$
1\otimes\bar{\phi}\colon
R(\tilde{T})\otimes_{R(T)}\ca{D}\to\tilde{\ca{D}},\qquad
1\otimes\bar{\phi}\colon R(\tilde{T})\otimes_{R(T)}I(\ca{D})\to
I(\tilde{\ca{D}}).
$$
\end{lemma}

\begin{proof}
It follows from Theorem \ref{theorem;operators} that
$\bar{\phi}(\partial_w)=\tilde{\partial}_w$ and
$\bar{\phi}(\partial'_w)=\tilde{\partial}'_w$, where
$\tilde{\partial}_w$ (resp.\ $\tilde{\partial}'_w$) is the operator on
$R(\tilde{T})$ analogous to $\partial_w$ (resp.\ $\partial'_w$), so
the statement follows from Theorem \ref{theorem;kostant-kumar}.
\end{proof}

\section{$\ca{D}$-modules}\label{section;module}

We shall see in Section \ref{section;push-pull} that the
$T$-equivariant K-group of a $G$-space is a module over the ring
$\ca{D}$.  In this section we collect some facts regarding abstract
$\ca{D}$-modules.  The notation is as in Sections
\ref{section;operator} and \ref{section;algebra}.

Let $A$ be a left $\ca{D}$-module.  We say an element of $A$ is
\emph{$\ca{D}$-invariant} or \emph{Hecke invariant} if it is
annihilated by all operators in the augmentation left ideal
$I(\ca{D})$.  We denote by $A^{I(\ca{D})}$ the group of invariants.
By Theorem
\ref{theorem;kostant-kumar}\eqref{item;kostant-kumar-ideal},
$$
A^{I(\ca{D})}=\{\,a\in A\mid\partial'_w(a)=0\text{ for all
$w\ne1$}\,\}.
$$

This group is not a $\ca{D}$-submodule of $A$, but a submodule over
the ring $R(T)^W\cong R(G)$.  Since $I(\ca{D})$ contains the
augmentation ideal $I(W)$ of $\Z[W]$, the Hecke invariants are
contained in the Weyl invariants,
\begin{equation}\label{equation;augmentation-weyl}
A^{I(\ca{D})}\subset A^W,
\end{equation}
an inclusion which is in general not an equality.  On the other hand,
for the $\ca{D}$-module $R(T)$ it follows from the definition
\eqref{equation;operator'} of the operators $\delta'_\alpha$ that
$$R(T)^{I(\ca{D})}=R(T)^W.$$
This leads to the following characterization of the augmentation left
ideal.

\begin{lemma}\label{lemma;augmentation}
$I(\ca{D})=\{\,\Delta\in\ca{D}\mid\Delta\circ\partial_{w_0}=0\,\}$,
the left annihilator of $\partial_{w_0}$.
\end{lemma}

\begin{proof}
Let $\Delta\in I(\ca{D})$ and $u\in R(T)$.  It follows from
\eqref{equation;long-symmetric} that $\partial_{w_0}(u)=j^*j_*(u)$ is
in $R(T)^W=R(T)^{I(\ca{D})}$ , so $\Delta\partial_{w_0}(u)=0$.  Hence
$\Delta\circ\partial_{w_0}=0$.  Conversely, let $\Delta\in\ca{D}$ and
suppose $\Delta\circ\partial_{w_0}=0$.  It follows from
\eqref{equation;push-pull-weyl} and \eqref{equation;long-symmetric}
that $\partial_{w_0}(1)=1$.  Therefore
$\Delta(1)=\Delta\partial_{w_0}(1)=0$, and hence $\Delta\in
I(\ca{D})$.
\end{proof}

\begin{lemma}\label{lemma;invariant}
Let $A$ be a $\ca{D}$-module.  Then $\Delta(ua)=\Delta(u)a$ for all
$\Delta\in\ca{D}$, $u\in R(T)$ and $a\in A^{I(\ca{D})}$.
\end{lemma}

\begin{proof}
The identity \eqref{equation;product} holds if we replace $u_1$ with
$u\in R(T)$ and $u_2$ with $a\in A$.  Therefore, for each root
$\alpha$, we have $\delta_\alpha(ua)=\delta_\alpha(u)a$ if
$\delta'_\alpha(a)=0$.  Let $a\in A^{I(\ca{D})}$.  Then
$\delta'_\alpha(a)=0$ for all $\alpha$, so
$\delta_\alpha(ua)=\delta_\alpha(u)a$ for all $\alpha$, and hence
$\partial_w(ua)=\partial_w(u)a$ for all $w\in W$ by Theorem
\ref{theorem;operators}.  It now follows from Theorem
\ref{theorem;kostant-kumar}\eqref{item;kostant-kumar} that
$\Delta(ua)=\Delta(u)a$ for all $\Delta\in\ca{D}$.
\end{proof}

The next two results give a measure of the discrepancy between $A^W$
and $A^{I(\ca{D})}$.  In Theorem \ref{theorem;weyl}, Lemma
\ref{lemma;left-inverse}\eqref{item;left-inverse} will be interpreted
as a ``Weyl character formula''.

\begin{lemma}\label{lemma;left-inverse}
Choose $u_0\in R(T)$ satisfying $\partial_{w_0}(u_0)=1$, e.g.\
$u_0=1$.  Define $\pi\in\ca{D}$ by $\pi(u)=\partial_{w_0}(u_0u)$ for
$u\in R(T)$.  Let
$$\ca{J}=\ca{D}\cdot\pi+\sum_{w\in W}\ca{D}\cdot(1-w)$$
be the left ideal of $\ca{D}$ generated by $\pi$ and $I(W)$.  Let $A$
be a left $\ca{D}$-module.
\begin{enumerate}
\item\label{item;left-inverse}
$\pi\colon A\to A$ projects $A$ onto the $R(T)^W$-submodule
$A^{I(\ca{D})}$.
\item\label{item;left-ideal}
$A^W=A^{I(\ca{D})}\oplus A^{\ca{J}}$.  Hence $A^W=A^{I(\ca{D})}$ if
and only if $A^{\ca{J}}=0$.
\end{enumerate}
\end{lemma}

\begin{proof}
It follows from Lemma \ref{lemma;augmentation} that $\pi$ maps $A$ to
$A^{I(\ca{D})}$.  If $a\in A^{I(\ca{D})}$, then Lemma
\ref{lemma;invariant} shows that
$$
\pi(a)=\partial_{w_0}(u_0a)=\partial_{w_0}(u_0)a=a,
$$
since $\partial_{w_0}(u_0)=1$.  This proves that
$\pi(A)=A^{I(\ca{D})}$ and $\pi^2=\pi$, which establishes
\eqref{item;left-inverse}.  It follows from \eqref{item;left-inverse}
that $A$ is the direct sum of the $R(T)^W$-submodules $A^{I(\ca{D})}$
and $\ker\pi$.  Moreover, it follows from \eqref{item;left-inverse}
and \eqref{equation;augmentation-weyl} that $\pi$ maps $A^W$ into
itself.  Therefore $A^W$ is the direct sum of $A^{I(\ca{D})}$ and
$A^W\cap\ker\pi=A^{\ca{J}}$.
\end{proof}

\begin{remark}\label{remark;extension}
Let $\k$ be an arbitrary commutative ring.  For a $\Z$-module $M$,
denote by $M_\k$ the $\k$-module $\k\otimes_\Z M$.  Lemma
\ref{lemma;left-inverse} generalizes in an obvious way if we extend
scalars from $\Z$ to $\k$.  Thus, if we take $u_0$ to be any element
of $R(T)_\k$ satisfying $\partial_{w_0}(u_0)=1$ and define
$\pi\in\ca{D}_\k$ by $\pi(u)=\partial_{w_0}(u_0u)$ for $u\in R(T)_\k$,
then $\pi$ projects $A$ onto $A^{I(\ca{D}_\k)}$ for any left
$\ca{D}_\k$-module $A$.
\end{remark}

Recall that $\d=\prod_{\alpha\in\ca{R}^+}(1-e^{-\alpha})\in R(T)$ is
the Weyl denominator \eqref{equation;denominator}.  For simplicity we
denote its image $1_\k\otimes\d$ in $R(T)_\k$ also by $\d$.

\begin{lemma}\label{lemma;discriminant}
The notation and the assumptions are as in Remark
{\rm\ref{remark;extension}}.
\begin{enumerate}
\item\label{item;weyl}
We have $\d\in\ca{J}_\k$.  Hence $A^W=A^{I(\ca{D}_\k)}$ if $\d$ is not
a zero divisor in $A$.
\item\label{item;weyl-invertible}
Assume that $\abs{W}$ is invertible in $\k$ and let
$u_0=\abs{W}^{-1}\d\in R(T)_\k$.  Then $\partial_{w_0}(u_0)=1$ and
$\pi\colon A\to A$ is the projection map onto the $W$-invariants,
$$\pi=\frac1{\abs{W}}\sum_{w\in W}w.$$
Hence $A^W=A^{I(\ca{D}_\k)}$.
\end{enumerate}
\end{lemma}

\begin{proof}
The $\rho$-shifted Weyl action $w\odot u=e^{-\rho}w(e^\rho u)$ on
$R(T)$ has the evident properties
\begin{equation}\label{equation;shift}
w\odot(u_1u_2)=(w\odot u_1)w(u_2),\qquad w\odot J(u)=J(w\odot
u)=\det(w)J(u)
\end{equation}
for all $w\in W$ and $u_1$, $u_2$, $u\in R(T)$, where $J$ is the
$\rho$-shifted antisymmetrizer \eqref{equation;anti-invariant}.  It
follows from \eqref{equation;long-symmetric} and
\eqref{equation;shift} that for all $u\in R(T)$
\begin{align*}
\d\pi(u)&=\d\partial_{w_0}(u_0u)=J(u_0u)=J(u_0)u+J(u_0u)-J(u_0)u\\
&=\d\partial_{w_0}(u_0)u+\sum_w\det(w)(w\odot u_0)(w-1)(u)\\
&=\d u+\sum_w\det(w)(w\odot u_0)(w-1)(u).
\end{align*}
Hence 
$$\d=\d\pi+\sum_w\det(w)(w\odot u_0)(1-w)\in\ca{J}.$$
Therefore, if $\d$ is not a zero divisor in $A$, then $A^{\ca{J}}=0$,
so $A^W=A^{I(\ca{D})}$ by Lemma
\ref{lemma;left-inverse}\eqref{item;left-ideal}.  The same argument
holds if we extend scalars from $\Z$ to $\k$.  This proves
\eqref{item;weyl}.  Now assume that $\abs{W}$ is invertible in $\k$.
It follows from \eqref{equation;shift} that $w\odot\d=\det(w)\d$ for
all $w\in W$.  Hence, by \eqref{equation;long-symmetric} and
\eqref{equation;shift},
\begin{multline*}
\partial_{w_0}(\d u)=\frac1{\d}J(\d u)=\frac1\d\sum_w\det(w)w\odot(\d
u)\\
=\frac1\d\sum_w\det(w)(w\odot\d)w(u)=\sum_ww(u)
\end{multline*}
for all $u\in R(T)_\k$.  In the first place, this identity shows that
$\partial_{w_0}(\d)=\abs{W}$, and so $\partial_{w_0}(u_0)=1$.
Secondly, it shows that
$$\pi(u)=\partial_{w_0}(u_0u)=\frac1{\abs{W}}\sum_{w\in W}w(u).$$
Hence $A^W=A^{I(\ca{D}_\k)}$ by virtue of Remark
\ref{remark;extension}.
\end{proof}

\section{Push-pull operators}\label{section;push-pull}

As in the previous section, $G$ denotes a compact connected Lie group
with maximal torus $T$, root system $\ca{R}$ and Weyl group $W$.  We
fix a basis of $\ca{R}$ and define the operators $\partial_w$ and
$\partial'_w$ as in Theorem \ref{theorem;operators}.  We denote by $X$
a compact topological space on which $G$ acts continuously.

\subsection*{The algebra $\ca{D}$ acts on K-theory}

As in \cite{segal;equivariant-k-theory;;1968} we denote by
$$K_G^*(X)=K_G^0(X)\oplus K_G^{-1}(X)$$
the equivariant K-theory of $X$.  Recall that
$$K_G^*(\pt)=K_G^0(\pt)=R(G)$$
is the representation ring of $G$, where $\pt$ denotes a space
consisting of a single point.

Let $j\colon T\to G$ be the inclusion map and
$$j^*\colon K_G^*(X)\longto K_T^*(X)$$
the restriction homomorphism.  As shown by Atiyah
\cite[\S\,4]{atiyah;bott-periodicity-index}, the functional $j_*\in
R(T)^\vee$ generalizes to a pushforward homomorphism
$$j_*\colon K_T^*(X)\longto K_G^*(X).$$
He proved that $j_*$ is $K_G^*(X)$-linear and satisfies $j_*(1)=1$,
and deduced from this the following ``splitting principle''.

\begin{proposition}%
[{\cite[Proposition~4.9]{atiyah;bott-periodicity-index}}]
\label{proposition;atiyah}
$j_*$ is a left inverse of $j^*$.  Hence $j^*$ is injective and maps
$K_G^*(X)$ onto a direct summand of $K_T^*(X)$.
\end{proposition}
 
\begin{remark}\label{remark;paracompact}
The compactness assumption on $X$ is frequently a nuisance in
practice.  A map $j_*$ with the above properties can presumably be
defined for any paracompact $G$-space $X$, provided that we replace
the compactly supported K-theory of
\cite{segal;equivariant-k-theory;;1968} with the K-theory of
\cite[\S\,4]{atiyah-segal;equivariant-completion}.  If this is true,
and if the K\"unneth formula of Theorem \ref{theorem;kunneth} below
generalizes to this setting, then it is a straightforward exercise to
extend the results of Sections \ref{section;push-pull} and
\ref{section;duality} to paracompact spaces.
\end{remark}

The purpose of this section is to prove Theorem \ref{theorem;d} below,
which describes the subgroup $j^*(K_G^*(X))$ in terms of divided
difference operators.  This subgroup is contained in, but not
necessarily equal to, $K_T^*(X)^W$, the Weyl invariants in $K_T^*(X)$.
(See Example \ref{example;mcleod}.)

First we show how to extend the Demazure operators to K-theory.
Prompted by \eqref{equation;long-symmetric}, we define the
endomorphism $\delta_X$ of $K_T^*(X)$ by
\begin{equation}\label{equation;push-pull}
\delta_X=j^*j_*.
\end{equation}
Similarly, let $\alpha$ be a root, let $G_\alpha$ be the centralizer
in $G$ of $\ker\alpha$, and let $j_\alpha\colon T\to G_\alpha$ be the
inclusion.  Define the endomorphism $\delta_{\alpha,X}$ of $K_T^*(X)$
by
\begin{equation}\label{equation;push-pull-root}
\delta_{\alpha,X}=j_\alpha^*j_{\alpha,*}.
\end{equation}
To define the induction map $j_{\alpha,*}$ we must choose a complex
structure on $G_\alpha/T$.  We do this by identifying $G_\alpha/T$
with the complex homogeneous space $(G_\alpha)_\C/B$, where $B$ is the
Borel subgroup of $(G_\alpha)_\C$ generated by $T_\C$ and the root
space $\g_\C^{-\alpha}$.

\begin{lemma}\label{lemma;square}
The operator $\delta_X$ is $K_G^*(X)$-linear:
$\delta_X(j^*(b)a)=j^*(b)\delta_X(a)$ for all $a\in K_T^*(X)$ and
$b\in K_G^*(X)$.  Hence $\delta_X(j^*(b))=j^*(b)$ and
$\delta_X^2=\delta_X$.
\end{lemma}

\begin{proof}
Since $j_*$ is a homomorphism of $K_G^*(X)$-modules,
$$\delta_X(j^*(b)a)=j^*j_*(j^*(b)a)=j^*(bj_*(a))=j^*(b)\delta_X(a)$$
for all $a\in K_T^*(X)$ and $b\in K_G^*(X)$.  Therefore, since
$j_*(1)=1$,
$$\delta_X(j^*(b))=j^*(b)\delta_X(1)=j^*(b)$$
and $\delta_X^2(a)=\delta_X(j^*j_*(a))=j^*j_*(a)=\delta_X(a)$.
\end{proof}

The operators $\delta_{\alpha,X}$ have the same properties, as one
sees by applying Lemma \ref{lemma;square} to the group $G=G_\alpha$.

To check that the $\delta_{\alpha,X}$ generate an action of the
algebra $\ca{D}$ on $K_T^*(X)$, we invoke a special case of the
equivariant K\"unneth theorem, which was established by the combined
efforts of Hodgkin \cite{hodgkin;equivariant-kunneth}, Snaith
\cite{snaith;kunneth}, McLeod \cite{mcleod;kunneth-equivariant}, and
Rosenberg and Schochet \cite{rosenberg-schochet;equivariant-k-theory}.

\begin{theorem}\label{theorem;kunneth}
Assume that $\pi_1(G)$ is torsion-free.  Then the map
$$R(T)\otimes_{R(G)}K_G^*(X)\to K_T^*(X)$$
defined by $u\otimes b\mapsto u\cdot j^*(b)$ is an isomorphism of
$\Z/2\Z$-graded $R(T)$-algebras.
\end{theorem}

Let $\ca{E}(X)$ be the ring of $K_G^*(X)$-linear endomorphisms of the
abelian group $K_T^*(X)$.  Let $R(T)\to\ca{E}(X)$ be the homomorphism
defining the natural $R(T)$-module structure on $K_T^*(X)$.  Let
$\ca{S}$ be the set $\{\,\delta_\alpha\mid\alpha\in\ca{R}\,\}$ and
define a map from $\ca{S}$ to $\ca{E}(X)$ by
$\delta_\alpha\mapsto\delta_{\alpha,X}$.  By definition, the set
$\ca{S}\cup R(T)$ generates the ring $\ca{D}$, and the following
statement says that the map $\ca{S}\cup R(T)\to\ca{E}(X)$ just defined
extends uniquely to a ring homomorphism $\ca{D}\to\ca{E}(X)$.

\begin{proposition}\label{proposition;lift}
The operators $\delta_{\alpha,X}$, for $\alpha\in\ca{R}$, together
with the natural $R(T)$-module structure define a unique
$\ca{D}$-module structure on $K_T^*(X)$.  This $\ca{D}$-module
structure commutes with the $K_G^*(X)$-module structure and is
contravariant with respect to $G$-equivariant continuous maps and
covering homomorphisms of $G$.
\end{proposition}

\begin{proof}
First assume that $\pi_1(G)$ is torsion-free.  Identify $K_T^*(X)$
with $R(T)\otimes_{R(G)}K_G^*(X)$ through the isomorphism of Theorem
\ref{theorem;kunneth}.  Then
$$\ca{E}(X)=\End_{K_G^*(X)}(K_T^*(X))=\ca{E}\otimes_{R(G)}K_G^*(X),$$
so the map $\ca{D}\to\ca{E}(X)$ defined by
$\Delta\mapsto\Delta\otimes1$, where $1$ is the identity automorphism
of $K_G^*(X)$, is a well-defined algebra homomorphism.  It follows
from \eqref{equation;long-symmetric} (applied to the group
$G=G_\alpha$) and from Lemma \ref{lemma;square} that
$\delta_{\alpha,X}=\delta_\alpha\otimes1$.  This proves that the
endomorphisms $\delta_{\alpha,X}$ generate a well-defined action of
$\ca{D}$ on $K_T^*(X)$.  If $\pi_1(G)$ is not torsion-free, we choose
a covering $\phi\colon\tilde{G}\to G$ of $G$ by a compact connected
$\tilde{G}$ such that $\pi_1(\tilde{G})$ is torsion-free.  (For
instance, we can take $\tilde{G}$ to be the direct product of a simply
connected group and a torus.)  According to a result of Snaith
\cite[Lemma~2.4]{snaith;kunneth}, the pullback map
$$
\phi^*\colon K_T^*(X)\longto K_{\tilde{T}}^*(X)
$$
is injective, where $\tilde{T}$ is the maximal torus $\phi^{-1}(T)$ of
$\tilde{G}$.  Let $\tilde{\delta}_{\alpha,X}
=\tilde{\jmath}_\alpha^*\tilde{\jmath}_{\alpha,*}$ be the operator on
$K_{\tilde{T}}^*(X)$ corresponding to $\alpha$, where
$\tilde{\jmath}_\alpha\colon\tilde{T}\to\tilde{G}_\alpha$ is the
inclusion.  It follows from the naturality properties of $j_\alpha^*$
and $j_{\alpha,*}$ that
\begin{equation}\label{equation;covering}
\phi^*\delta_{\alpha,X}=\tilde{\delta}_{\alpha,X}\phi^*.
\end{equation}
Recall from Lemma \ref{lemma;covering} that $\phi$ induces an
injective algebra homomorphism
$$
\bar{\phi}\colon\ca{D}=\bigoplus_wR(T)\partial_w\longto\tilde{\ca{D}}
=\bigoplus_wR(\tilde{T})\tilde{\partial}_w.
$$
We already know that the $\tilde{\delta}_{\alpha,X}$ generate a
well-defined $\tilde{\ca{D}}$-action on $K_{\tilde{T}}^*(X)$.  The
restriction of this action to the subalgebra $\ca{D}$ preserves the
submodule $K_T^*(X)$ and, because of \eqref{equation;covering}, the
elements $\delta_\alpha$ act in the required fashion.  The
$\ca{D}$-module structure on $K_T^*(X)$ so defined is unique because
the $\delta_\alpha$ and $R(T)$ generate the ring $\ca{D}$, and it
commutes with the $K_G^*(X)$-module structure because of Lemma
\ref{lemma;square}.  The naturality properties with respect to
equivariant maps and covering homomorphisms follow from the
corresponding properties of $j^*$ and $j_*$.
\end{proof}

We will write the product (with respect to the module structure given
by the proposition) of a class $a\in K_T^*(X)$ by an operator
$\Delta\in\ca{D}$ as $\Delta_X(a)$, or as $\Delta(a)$ when there is no
danger of ambiguity.  The formul{\ae} \eqref{equation;operator} and
\eqref{equation;operator'} translate into the identities
$$
s_\alpha=e^\alpha-(e^\alpha-1)\delta_\alpha
=1-(1-e^{-\alpha})\delta'_\alpha
$$
of operators on $K_T^*(X)$.

\subsection*{Hecke invariants versus Weyl invariants}

The next result means that a $T$-equi\-variant K-class on $X$ is
$G$-equi\-vari\-ant if and only if it is annihilated by the operators
$\partial'_w$ for all $w\ne1$.

\begin{theorem}\label{theorem;d}
The map $j^*$ is an isomorphism from $K_G^*(X)$ onto
$K_T^*(X)^{I(\ca{D})}$.
\end{theorem}

\begin{proof}
First assume that $\pi_1(G)$ is torsion-free.  Let
$\lie{A}=\ca{D}$-$\lie{Mod}$ and $\lie{B}=R(G)$-$\lie{Mod}$ be the
categories of (left) modules over the rings $\ca{D}$, resp.\ $R(G)$.
By the Pittie-Steinberg theorem, the $R(G)$-module $R(T)$ is free, and
therefore it is a progenerator of the category $\lie{B}$.  By
Proposition \ref{proposition;endomorphism}\eqref{item;algebras},
$\ca{D}$ is the full endomorphism algebra of $R(T)$.  Hence, by the
first Morita equivalence theorem (see e.g.\
\cite[\S\,18]{lam;modules-rings}), the functor
$\lie{G}\colon\lie{B}\to\lie{A}$ defined by
$$B\longmapsto R(T)\otimes_{R(G)}B$$
is an equivalence with inverse $\lie{F}\colon\lie{A}\to\lie{B}$ given
by
$$A\longmapsto\Hom_{\ca{D}}(R(T),A).$$
We assert that $\lie{F}$ is naturally isomorphic to the functor
$\lie{I}\colon\lie{A}\to\lie{B}$ given
by
$$A\longmapsto A^{I(\ca{D})}.$$
Indeed, define the natural $R(G)$-linear map
$\Phi_A\colon\lie{F}(A)\to A$ by $\Phi_A(f)=f(1)$.  The map $\Phi_A$
is injective because $\ca{D}\supset R(T)$.  By $\ca{D}$-linearity,
$$\Delta(\Phi_A(f))=\Delta(f(1))=f(\Delta(1))=0$$
for all $f\in\lie{F}(A)$ and $\Delta\in\ca{D}$, so $\Phi_A(f)\in
\lie{I}(A)$.  For $a\in\lie{I}(A)$ define $f_a\colon R(T)\to A$ by
$f_a(u)=ua$.  It follows from Lemma \ref{lemma;invariant} that
$f_a\in\lie{F}(A)$, and clearly $\Phi_A(f_a)=a$.  Therefore the image
of $\Phi_A$ is equal to $\lie{I}(A)$.  We conclude that $\Phi$ is a
natural isomorphism from $\lie{F}$ to $\lie{I}$.  Now consider the
$\ca{D}$-module $A=K_T^*(X)$ and the $R(G)$-module $B=K_G^*(X)$.  By
Theorem \ref{theorem;kunneth}, $A\cong\lie{G}(B)$.  Hence
$$B\cong\lie{F}(A)\cong\lie{I}(A)=A^{I(\ca{D})}.$$
If $\pi_1(G)$ is not torsion-free, we choose a covering group
$\phi\colon\tilde{G}\to G$ as in the proof of Proposition
\ref{proposition;lift} and consider the diagram (cf.\
\cite[\S\,4]{mcleod;kunneth-equivariant})
$$
\xymatrix{
K_T^*(X)\ar[r]^{\phi^*}\ar@<0.5ex>[d]^{j_*}&
K_{\tilde{T}}^*(X)\ar@<0.5ex>[d]^{\tilde{\jmath}_*}\\
K_G^*(X)\ar[r]^{\phi^*}\ar@<0.5ex>[u]^{j^*}&
K_{\tilde{G}}^*(X)\ar@<0.5ex>[u]^{\tilde{\jmath}^*},
}
$$
where we have $\phi^*\circ j^*=\tilde{\jmath}^*\circ\phi^*$ and
$\tilde{\jmath}_*\circ\phi^*=\phi^*\circ j_*$.  We know that $j_*$ is
a left inverse of $j^*$ (Proposition \ref{proposition;atiyah}), that
the image of $\tilde{\jmath}^*$ is the submodule
$\tilde{M}=K_{\tilde{T}}^*(X)^{I(\tilde{\ca{D}})}$ of
$K_{\tilde{T}}^*(X)$ and that $\tilde{\jmath}_*|\tilde{M}$ is a
two-sided inverse of $\tilde{\jmath}^*\colon
K_{\tilde{G}}^*(X)\to\tilde{M}$.  It follows from Lemma
\ref{lemma;covering} that $\phi^*$ maps $M=K_T^*(X)^{I(\ca{D})}$ to
$\tilde{M}$.  As noted in the proof of Proposition
\ref{proposition;lift}, the map $K_T^*(X)\to K_{\tilde{T}}^*(X)$ is
injective, and hence so is the map $K_G^*(X)\to K_{\tilde{G}}^*(X)$.
An easy diagram chase now shows that $j^*(K_G^*(X))=M$.
\end{proof}

An immediate consequence is the following statement, which for a space
consisting of a single point amounts to the Weyl character formula.

\begin{theorem}\label{theorem;weyl}
For every $a\in K_T^*(X)$ there exists a unique $b\in K_G^*(X)$ such
that $j^*(b)=\partial_{w_0}(a)$.
\end{theorem}

\begin{proof}
It follows from \eqref{equation;long-symmetric} that $\partial_{w_0}$
acts on $K_T^*(X)$ as the operator $\delta_X$ defined in
\eqref{equation;push-pull}.  Taking $A=K_T^*(X)$ and $u_0=1$ in Lemma
\ref{lemma;left-inverse}\eqref{item;left-inverse}, we see that
$\partial_{w_0}$ projects $K_T^*(X)$ onto $K_T^*(X)^{I(\ca{D})}$.  Now
let $a\in K_T^*(X)$.  Then, by Theorem \ref{theorem;d}, $j_*(a)$ is
the unique class $b\in K_G^*(X)$ satisfying
$j^*(b)=\partial_{w_0}(a)$.
\end{proof}

\begin{example}\label{example;mcleod}
This example, which generalizes
\cite[Remark~4.5]{mcleod;kunneth-equivariant}, shows that in general
$j^*$ is not an isomorphism onto $K_T^*(X)^W$, not even if $G$ is
simply connected.  Let $G=\SU(2)$ and let $X$ be any $G$-space.  Let
$T$ be the diagonal maximal torus of $G$ and let $\varpi\in\X(T)$ be
the fundamental weight of $G$ defined by
$\varpi\bigl(\mskip-\thinmuskip
\begin{smallmatrix}t&0\\0&t^{-1}\end{smallmatrix}
\mskip-\thinmuskip\bigr)=t$.  Then $\alpha=2\varpi$ is the
corresponding simple root, and
$$R(T)\cong\Z[x,x^{-1}],\qquad R(G)\cong\Z[x+x^{-1}],$$
where $x=e^\varpi$.  The elements $1$ and $x$ are a basis of the
$R(G)$-module $R(T)$.  Relative to this basis the simple reflection
$s_\alpha=w_0$ and the operator $\delta'_\alpha=\partial'_{w_0}$ are
given by the matrices
\begin{equation}\label{equation;matrices}
s_\alpha=\begin{pmatrix}1&x+x^{-1}\\0&-1\end{pmatrix},\qquad
\delta'_\alpha=\begin{pmatrix}0&0\\0&1\end{pmatrix}.
\end{equation}
By the K\"unneth theorem,
$$K_T^*(X)\cong R(T)\otimes_{R(G)}K_G^*(X),$$
so the classes $1\otimes1$ and $x\otimes1$ form a basis of the
$K_G^*(X)$-module $K_T^*(X)$, relative to which $s_\alpha$ and
$\delta'_\alpha$ act again as the matrices \eqref{equation;matrices}.
Expressing an arbitrary class $a\in K_T^*(X)$ as $a=b_1+b_2x$ with
$b_1$, $b_2\in K_G^*(X)$, we find by straightforward calculation
\begin{gather*}
K_T^*(X)^W=\{\,b_1+b_2x\mid2b_2=(x+x^{-1})b_2=0\,\}
=K_G^*(X)\cdot1\oplus K_G^*(X)^I\cdot x,\\
K_T^*(X)^{I(\ca{D})}=\{\,b_1+b_2x\mid b_2=0\,\}=K_G^*(X)\cdot1.
\end{gather*}
Here $I$ denotes the ideal of $R(G)$ generated by $2$ and $x+x^{-1}$,
and $K_G^*(X)^I$ denotes the $R(G)$-submodule of $K_G^*(X)$
annihilated by $I$.  Thus $K_T^*(X)^W=j^*(K_G^*(X))$ if and only if
$K_G^*(X)$ contains no $I$-torsion elements.  In particular, if $X$ is
a principal $G$-bundle with base $Y=X/G$ such that $K^*(Y)$ has
$2$-torsion (e.g.\ $Y=\P^n(\R)$ and $X=G\times Y$), then
$K_T^*(X)^W\ne j^*(K_G^*(X))$.
\end{example}

We close this section by stating some criteria for $K_G^*(X)$ to be
isomorphic to $K_T^*(X)^W$.

\begin{theorem}\label{theorem;abelian}
\begin{enumerate}
\item\label{item;zero-divisor}
Assume that the Weyl denominator
$\d=\prod_{\alpha\in\ca{R}^+}(1-e^{-\alpha})$ is not a zero divisor in
the $R(T)$-module $K_T^*(X)$.  Then $j^*$ is an isomorphism from
$K_G^*(X)$ onto $K_T^*(X)^W$.
\item\label{item;weyl-order}
Let $\k$ be a commutative ring in which $\abs{W}$ is invertible, such
as $\Z\bigl[\abs{W}^{-1}\bigr]$ or $\Q$.  Then $j^*$ is an isomorphism
from $K_G^*(X)\otimes\k$ onto $(K_T^*(X)\otimes\k)^W$.
\end{enumerate}
\end{theorem}

\begin{proof}
This follows immediately from Theorem \ref{theorem;d} and Lemma
\ref{lemma;discriminant}.
\end{proof}

\begin{corollary}\label{corollary;abelian-torsion}
The map $j^*\colon K_G^*(X)\to K_T^*(X)^W$ is an isomorphism in each
of the following cases.
\begin{enumerate}
\item\label{item;torsion-free}
$K_T^*(X)$ is a torsion-free $R(T)$-module.
\item\label{item;injective}
The restriction homomorphism $K_T^*(X)\to K_T^*(X^T)$ is injective.
\item\label{item;symplectic}
$X$ is a compact Hamiltonian $G$-manifold.
\item\label{item;projective}
$X$ is a nonsingular complex projective variety on which $G$ acts by
linear transformations.
\end{enumerate}
\end{corollary}

\begin{proof}
Case \eqref{item;torsion-free} follows immediately from Theorem
\ref{theorem;abelian}\eqref{item;zero-divisor}.  Let $A=K^*(X^T)$.
Then the algebra $K_T^*(X^T)$ is isomorphic to the group algebra
$$K_T^*(X^T)\cong A\otimes_\Z R(T)\cong A[\X(T)].$$
It follows from \cite[\S\,VI.3.2]{bourbaki;groupes-algebres} that
$1-e^\alpha$ is not a zero divisor in $A[\X(T)]$ for any root
$\alpha$.  Therefore case \eqref{item;injective} is also a consequence
of Theorem \ref{theorem;abelian}\eqref{item;zero-divisor}.  According
to \cite[Theorem~2.5]{harada-landweber;k-theory-abelian-symplectic},
the assumption of case \eqref{item;injective} is satisfied in case
\eqref{item;symplectic}, so case \eqref{item;symplectic} follows from
case \eqref{item;injective}.  Finally, case \eqref{item;projective}
follows from case \eqref{item;symplectic}, because a nonsingular
subvariety of $\P^n(\C)$ which is stable under the action of a
subgroup $G$ of $\U(n+1)$ is a Hamiltonian $G$-manifold.  (See e.g.\
\cite[\S\,2]{kirwan;cohomology-quotients-symplectic}.)
\end{proof}

\begin{example}\label{example;flag}
Let $X$ be the flag variety $G/T$.  Then $K_G^*(X)\cong K_{G\times
T}^*(G)\cong R(T)$.  To compute $K_T^*(X)$, let us assume that
$\pi_1(G)$ is torsion-free.  Then, by the K\"unneth theorem,
\begin{equation}\label{equation;flag}
K_T^*(X)\cong R(T)\otimes_{R(G)}K_G^*(X)\cong
R(T)\otimes_{R(G)}R(T)\cong\bigl(R(T)\otimes_\Z R(T)\bigr)\big/I,
\end{equation}
where $I$ is the ideal of $R(T)\otimes_\Z R(T)$ generated by all
elements of the form $b\otimes1-1\otimes b$ with $b\in R(G)$.  The
isomorphism \eqref{equation;flag} is Weyl equivariant with respect to
the $W$-action on $R(T)\otimes_\Z R(T)$ given by $w(a_1\otimes
a_2)=w(a_1)\otimes a_2$.  The inclusion $j^*\colon K_G^*(X)\to
K_T^*(X)$ is induced by the map $R(T)\to R(T)\otimes_\Z R(T)$ which
sends $a$ to $1\otimes a$.  We now conclude from \eqref{equation;flag}
and the fact that $R(T)$ is a free $R(G)$-module that
$$
K_T^*(X)^W\cong\bigl(R(T)\otimes_{R(G)}R(T)\bigr)^W\cong
R(T)^W\otimes_{R(G)}R(T)\cong R(T)\cong K_G^*(X).
$$
Since the flag variety is complex projective, this follows also from
Corollary \ref{corollary;abelian-torsion}\eqref{item;projective}.  The
isomorphism \eqref{equation;flag} is $\ca{D}$-linear with respect to
the $\ca{D}$-action on $R(T)\otimes_\Z R(T)$ given by
$\Delta(a_1\otimes a_2)=\Delta(a_1)\otimes a_2$ (which preserves the
ideal $I$).  Therefore the projection onto the Weyl invariants
\begin{equation}\label{equation;projection}
\partial_{w_0}\colon R(T)\otimes_{R(G)}R(T)\longto R(T)
\end{equation}
is given by $\partial_{w_0}(a_1\otimes a_2)=\partial_{w_0}(a_1)a_2$.
\end{example}

\begin{example}\label{example;su2}
As a special case of Example \ref{example;flag}, consider $G=\SU(2)$.
Then $X=\P^1(\C)$.  Choose the maximal torus $T$ and the fundamental
weight $\varpi$ as in Example \ref{example;mcleod}.  Then
$R(T)\otimes_\Z R(T)\cong\Z[x,y,(xy)^{-1}]$, where $x$ and $y$ are two
copies of the generator $e^\varpi$ of $R(T)$.  The isomorphism
\eqref{equation;flag} allows us to think of classes in $K_T^*(X)$ as
cosets in $\Z[x,y,(xy)^{-1}]$ of the ideal $I$.  Under this
identification, the ring $K_G^*(X)\cong K_T^*(X)^W$ corresponds to the
subring $\Z[y,y^{-1}]$ of $\Z[x,y,(xy)^{-1}]/I$.  By definition, $I$
is generated by all elements of the form $f(x)-f(y)$, where $f\in
R(T)^W$ is a Weyl invariant Laurent polynomial in one variable.  It
follows that $I$ is generated by the single element
$x+x^{-1}-(y+y^{-1})$.  Using \eqref{equation;projection} we can
calculate the projection of the coset $x^ky^l+I$ onto $\Z[y,y^{-1}]$
for $k$, $l\in\Z$.  We have $\partial_{w_0}=\delta_\alpha$, where
$\alpha=2\varpi$ is the simple root, so
$$
\partial_{w_0}(x^ky^l+I)=\delta_\alpha(y^k)y^l=
\begin{cases}
(y^k+y^{k-2}+\cdots+y^{-k})y^l&\text{if $k\ge0$}\\
0&\text{if $k=-1$}\\
-(y^{k+2}+y^{k+4}+\cdots+y^{-k-2})y^l&\text{if $k\le-2$.}
\end{cases}
$$
\end{example}

\section{Relative duality}\label{section;duality}

As in Section \ref{section;push-pull}, $G$ denotes a compact connected
Lie group with maximal torus $T$ and Weyl group $W$, and $X$ denotes a
compact $G$-space.  Choose a basis of the root system of $(G,T)$ and
let $j_*\colon K_T^*(X)\to K_G^*(X)$ be the corresponding pushforward
homomorphism.  The pairing $\ca{P}$ defined in
\eqref{equation;pairing} generalizes to a bi-additive pairing
$\ca{P}_X\colon K_T^*(X)\times K_T^*(X)\to K_G^*(X)$ defined by
$$\ca{P}_X(a_1,a_2)=j_*(a_1a_2)$$
for $a_1$, $a_2\in K_T^*(X)$.  This pairing is graded symmetric in the
sense that
$$\ca{P}_X(a_1,a_2)=(-1)^{k_1k_2}\ca{P}_X(a_2,a_1)$$
if $a_1$ is of degree $k_1$ and $a_2$ is of degree $k_2$.  It is
bilinear over $K_G^*(X)$ in the sense that
$$
\ca{P}_X(j^*(b)a_1,a_2)=b\ca{P}_X(a_1,a_2),\qquad
\ca{P}_X(a_1,a_2j^*(b))=\ca{P}_X(a_1,a_2)b
$$
for all $b\in K_G^*(X)$.  It follows from the naturality of $j_*$ that
\begin{equation}\label{equation;natural-pairing}
\ca{P}_X(f^*(a_1),f^*(a_2))=f^*\ca{P}_Y(a_1,a_2)
\end{equation} 
for $a_1$, $a_2\in K_T^*(Y)$, where $f\colon X\to Y$ is any
$G$-equivariant continuous map.

Identifying $K_T^*(X)$ with $K_G^*(X\times G/T)$, we can view the
pairing $\ca{P}_X$ as a fibrewise intersection product in
$G$-equivariant K-theory for the projection map $X\times G/T\to X$,
and the following proposition as a duality theorem for this map.

\begin{proposition}\label{proposition;orthogonal}
Assume that $\pi_1(G)$ is torsion-free.  Then the pairing $\ca{P}_X$
is nonsingular.  Hence
$$
K_T^*(X)\cong\Hom_{K_G^*(X)}(K_T^*(X),K_G^*(X))
$$
as $\Z/2\Z$-graded left $K_G^*(X)$-modules.
\end{proposition}

\begin{proof}
Let $(u_w)_{w\in W}$ be a basis of the $R(G)$-module $R(T)$ and
$(u^w)_{w\in W}$ its $\ca{P}$-dual basis, which is characterized by
$\ca{P}(u_w,u^{w'})=\delta_{w,w'}$ for all $w$, $w'\in W$.  (As
reviewed in Section \ref{section;operator}, such bases exist because
$\pi_1(G)$ is torsion-free.)  Let $p\colon X\to\pt$ be the constant
map and define $\bar{u}_w=p^*(u_w)$ and $\bar{u}^w=p^*(u^w)$.  By the
K\"unneth theorem, Theorem \ref{theorem;kunneth}, $(\bar{u}_w)_{w\in
W}$ is a basis of the $K_G^*(X)$-module $K_T^*(X)$.  It follows from
\eqref{equation;natural-pairing} that
$$
\ca{P}_X(\bar{u}_w,\bar{u}^{w'})=p^*\ca{P}(u_w,u^{w'})=\delta_{w,w'}
$$
for all $w$, $w'\in W$.  In other words, the tuple $(\bar{u}^w)_{w\in
W}$ is a basis of $K_T^*(X)$ which is dual to $(\bar{u}_w)_{w\in W}$
with respect to the pairing $\ca{P}_X$.
\end{proof}

\begin{example}\label{example;poincare}
Assume that $\pi_1(G)$ is torsion-free.  Taking $X=G$, equipped with
the left multiplication action of $G$, we find from Proposition
\ref{proposition;orthogonal} that the pairing
$$K^*(G/T)\times K^*(G/T)\to\Z$$
is nonsingular, a form of Poincar\'e duality for the flag variety
$G/T$.  Taking $X=G/T$, we find a $T$-equivariant Poincar\'e duality
theorem for the flag variety, namely the assertion that the pairing
$$K_T^*(G/T)\times K_T^*(G/T)\to R(T)$$
is nonsingular.  Taking $X=\pt$, we obtain $G$-equivariant Poincar\'e
duality, namely that the pairing
$$K_G^*(G/T)\times K_G^*(G/T)\to R(G)$$
is nonsingular (which is equivalent to the pairing $\ca{P}$ on $R(T)$
being nonsingular).
\end{example}

\section{Algebraic equivariant K-theory}\label{section;scheme}

The results of Sections \ref{section;push-pull} and
\ref{section;duality} extend to the setting of algebraic K-theory
thanks to the work of Thomason
\cite{thomason;algebraic-k-theory-group-scheme}, Panin
\cite{panin;algebraic-k-theory-twisted-flag-varieties} and Merkurjev
\cite{merkurjev;comparison-equivariant-ordinary,%
merkurjev;equivariant-k-theory}.  As the arguments are closely
parallel to those presented in the topological context, we will keep
our exposition brief.

Let $\k$ be a field.  By a \emph{variety} we will mean a
quasi-projective scheme over $\k$.  All morphisms, sheaves, algebraic
groups and their actions will be assumed to be defined over $\k$.  We
denote the point object $\Spec(\k)$ by $\pt$, and the unique morphism
from a variety $X$ to $\pt$ by $p$ or $p_X$.  We denote the group of
characters (defined over $\k$) of an algebraic group $H$ by
$\ca{X}(H)$.

Let $G$ be a split reductive group and $X$ a $G$-variety.  We denote
by $\lie{M}^G(X)$ the category of $G$-equivariant coherent
$\ca{O}_X$-modules, where $\ca{O}_X$ is the structure sheaf of $X$,
and by $K^G_*(X)$ Quillen's K-theory of $\lie{M}^G(X)$.  Similarly, we
denote by $\lie{P}^G(X)$ the category of $G$-equivariant locally free
coherent $\ca{O}_X$-modules and by $K_G^*(X)$ its K-theory.  We have
$\lie{M}^G(\pt)=\lie{P}^G(\pt)=\Rep(G)$, the category of $G$-modules
that are finite-dimensional over $\k$, and
$K^G_0(\pt)=K_G^0(\pt)=R(G)$, the Grothendieck ring of $\lie{Rep}(G)$.
The theory $K^G_*(X)$ is covariant with respect to projective
$G$-morphisms, contravariant with respect to flat $G$-morphisms, and
contravariant with respect to group homomorphisms.  The theory
$K_G^*(X)$ is contravariant with respect to arbitrary $G$-morphisms
and with respect to group homomorphisms.  Moreover, tensor
multiplication induces a ring structure on $K_G^0(X)$ and a
$K_G^0(X)$-module structure on $K^G_*(X)$.

Choose a Borel subgroup $B$ of $G$ and a split maximal torus $T$ of
$B$.  Let $j\colon T\to G$ and $k\colon T\to B$ denote the respective
inclusions.  These induce restriction maps
$$
j_X^*=j^*\colon K^G_*(X)\longto K^T_*(X),\qquad k_X^*=k^*\colon
K^B_*(X)\overset\cong\longto K^T_*(X),
$$
the second of which is an isomorphism by
\cite[Corollary~2.15]{merkurjev;comparison-equivariant-ordinary}.  The
flag variety $\ca{B}=G/B$ is projective, so the projection $\pr\colon
X\times \ca{B}\to X$ is a projective $G$-morphism.  It is also flat
and therefore induces homomorphisms
$$
\pr^*\colon K^G_*(X)\longto K^G_*(X\times\ca{B}),\qquad\pr_*\colon
K^G_*(X\times\ca{B})\longto K^G_*(X).
$$
The $B$-morphism $i\colon X\cong X\times\{B\}\to X\times\ca{B}$
induces an isomorphism
$$i^*\colon K^G_*(X\times\ca{B})\overset\cong\longto K^B_*(X)$$
by \cite[Proposition~2.10]{merkurjev;comparison-equivariant-ordinary}.
We have $j^*=k^*i^*\pr^*$ and we define a pushforward homomorphism by
$$
j_{X,*}=j_*=\pr_*(i^*)^{-1}(k^*)^{-1}\colon K^T_*(X)\longto K^G_*(X).
$$

The fact that the variety $\ca{B}$ is smooth and rational implies that
$j_*([\ca{O}_X])=[\ca{O}_X]$, where $[\ca{O}_X]$ is the class of the
structure sheaf in $K^T_0(X)$, resp.\ $K^G_0(X)$.  It now follows from
the projection formula, \cite[\S\,7,
Proposition~2.10]{quillen;higher-algebraic-k-theory-I}, that $j_*$ is
a left inverse of $j^*$ and that $j^*(K^G_*(X))$ is a direct summand
of $K^T_*(X)$.

We now define the endomorphisms $\delta_X=j^*j_*$ and
$\delta_{\alpha,X}=j_\alpha^*j_{\alpha,*}$ of $K^T_*(X)$ just as in
Section \ref{section;push-pull}.  Showing that these operators
generate a $\ca{D}$-action requires a number of auxiliary results.

Let $H$ be an algebraic group and let $Z$ be an $H\times B$-variety
such that the quotient $Z\to Z/B$ exists and is a Zariski locally
trivial principal $B$-bundle.  Let $E_\lambda$ be the one-dimensional
$B$-module defined by $\lambda\in\X(T)\cong\X(B)$, and let
$\ca{L}_\lambda\in\lie{M}^H(Z/B)$ be the sheaf of sections of the
$H$-equivariant line bundle $Z\times^BE_\lambda$ over $Z/B$.  As in
\cite[\S\,1]{demazure;desingularisation}, we define the
\emph{characteristic homomorphism}
$$c_Z^H=c^H\colon R(T)\longto K^H_0(Z/B)$$
by $c^H(e^\lambda)=[\ca{L}_\lambda]$.  Take $Z=\overline{BwB}$, the
closure of the double coset of $w\in W$ in $G$, with $B$ acting by
right multiplication, and take $H$ to be any closed subgroup of $G$
such that left multiplication by $H$ preserves $Z$.  Then we obtain a
homomorphism
$$c_w^H\colon R(T)\longto K^H_0(S_w),$$
where $S_w=\overline{BwB/B}\subset\ca{B}$ is the Schubert variety
corresponding to $w$.  (If $w=w_0$, then $S_w=\ca{B}$ and $H$ is
allowed to be any closed subgroup of $G$.)  The projective morphism
$p_w=p_{S_w}\colon S_w\to\pt$ induces a map
$$p_{w,*}^H\colon K^H_0(S_w)\longto K^H_0(\pt)\cong R(H).$$
Taking $H=T$ we find an endomorphism $p_{w,*}^T\circ c_w^T$ of $R(T)$,
which is equal to Demazure's operator $\partial_w$.

\begin{theorem}\label{theorem;demazure}
We have $\partial_w=p_{w,*}^T\circ c_w^T$ for all $w\in W$, and
$\partial_{w_0}=j_\pt^*\circ j_{\pt,*}$.
\end{theorem}

\begin{proof}
The first assertion is Demazure's character formula, \cite[\S\,5,
Th\'eor\`eme~2]{demazure;desingularisation},
\cite[Theorem~4.3]{andersen;schubert-demazure}.  Now let $w=w_0$, so
that $S_w=\ca{B}$.  Write $c_{w_0}^H=c^H$ for any closed subgroup $H$
of $G$ and $p_{w_0}=p$.  It follows from the naturality properties of
pullbacks and pushforwards
(\cite[\S\,7.2]{quillen;higher-algebraic-k-theory-I}) that the diagram
$$
\xymatrix@=2em{
R(T)\ar@/^/[drr]^{j_{\pt,*}}\ar[dr]|-{c^G}\ar@/_/[ddr]_{c^T}\\
&K^G_0(\ca{B})\ar[d]^{j_{\ca{B}}^*}\ar[r]_{p_*^G}&R(G)\ar[d]_{j_\pt^*}\\
&K^T_0(\ca{B})\ar[r]^{p_*^T}&R(T)
}
$$
commutes, and hence that $\partial_{w_0}=p_*^T\circ c^T=j_\pt^*\circ
j_{\pt,*}$.
\end{proof}

The next statement is analogous to a result of Snaith
\cite[Lemma~2.4]{snaith;kunneth} and generalizes a result of Uma
\cite[Lemma~1.7]{uma;equivariant-k-theory}.

\begin{lemma}\label{lemma;snaith-uma}
Let $\tilde{T}$ be a split torus and $\phi\colon\tilde{T}\to T$ an
isogeny.  Then, for every $T$-variety $X$, the map
$R(\tilde{T})\otimes_{R(T)}K^T_*(X)\to K^{\tilde{T}}_*(X)$ defined by
$u\otimes a\mapsto u\cdot\phi^*(a)$ is an isomorphism.  Hence the map
$\phi^*\colon K^T_*(X)\to K^{\tilde{T}}_*(X)$ is injective and has an
$R(T)$-linear left inverse.
\end{lemma}

\begin{proof}
Let $\psi\colon C\to\tilde{T}$ be the reduced kernel of $\phi$ and let
$\lie{M}$ be the abelian category
$$\lie{M}=\prod_{\lambda\in\X(C)}\lie{M}^T(X).$$
We write an object of $\lie{M}$ as an $\X(C)$-tuple
$(\ca{F}_\lambda)_{\lambda\in\X(C)}$, where each $\ca{F}_\lambda$ is a
$T$-equivariant coherent sheaf on $X$.  The restriction homomorphism
$\psi^*\colon\X(\tilde{T})\to\X(C)$ is surjective
(\cite[\S\,3.2]{springer;linear-algebraic-groups}); we choose a
(set-theoretic) left inverse $\sigma$.  Let $V_\xi$ be the
one-dimensional $\tilde{T}$-module defined by a character
$\xi\in\X(\tilde{T})$.  Let $\ca{G}$ be a $\tilde{T}$-equivariant
coherent sheaf on $X$.  Since $C$ acts trivially on $X$, the subsheaf
$\ca{G}^C$ of $C$-invariant sections of $\ca{G}$ is a well-defined
$T$-equivariant $\ca{O}_X$-module.  The group $C$, being
diagonalizable, is linearly reductive, so every locally finite
$C$-module is a sum of isotypical components.  (See e.g.\
\cite[\S\,III, \S\,V]{springer;aktionen-reduktiver-gruppen}.)  It
follows that the subsheaf $\ca{G}^C$ is coherent.  This enables us to
define functors
$$
\xymatrix{
{\lie{M}}\ar@<0.5ex>[r]^-\mu&
{\lie{M}^{\tilde{T}}(X)}\ar@<0.5ex>[l]^-\nu
}
$$
by
\begin{gather*}
\mu\bigl((\ca{F}_\lambda)_{\lambda\in\X(C)}\bigr)
=\bigoplus_{\lambda\in\X(C)}p^*(V_{\sigma(\lambda)})
\otimes_{\ca{O}_X}\phi^*(\ca{F}_\lambda),\\
\nu(\ca{G})=\bigl(\sheafhom_{\ca{O}_X}(p^*(V_{\sigma(\lambda)}),
\ca{G})^C\bigr)_{\lambda\in\X(C)}.
\end{gather*}
Here $\phi^*(\ca{F})$ denotes a $T$-equivariant sheaf $\ca{F}$
regarded as a $\tilde{T}$-equivariant sheaf via the homomorphism
$\phi$, and $\sheafhom_{\ca{O}_X}$ denotes sheaf hom.  Then, for all
objects $(\ca{F}_\lambda)_{\lambda\in\X(C)}$ of $\lie{M}$,
\begin{multline*}
\nu\mu((\ca{F}_\lambda)_\lambda)=\Bigl(\bigoplus_\lambda
\sheafhom_{\ca{O}_X}
\bigl(p^*(V_{\sigma(\lambda')}),p^*(V_{\sigma(\lambda)})
\otimes_{\ca{O}_X}\phi^*(\ca{F}_\lambda)\bigr)^C\Bigr)_{\lambda'}
\\
\cong\Bigl(\bigoplus_\lambda \sheafhom_{\ca{O}_X}
\bigl(p^*(V_{\sigma(\lambda')}),p^*(V_{\sigma(\lambda)})\bigr)^C
\otimes_{\ca{O}_X}\ca{F}_\lambda\Bigr)_{\lambda'}
\\
\cong\Bigl(\bigoplus_\lambda p^*\bigl(\Hom_\k
(V_{\sigma(\lambda')},V_{\sigma(\lambda)})^C\bigr)
\otimes_{\ca{O}_X}\ca{F}_\lambda\Bigr)_{\lambda'}
\cong(\ca{F}_\lambda)_\lambda,
\end{multline*}
and, for all objects $\ca{G}$ of $\lie{M}^{\tilde{T}}(X)$,
$$
\mu\nu(\ca{G})=\bigoplus_{\lambda\in\X(C)}p^*(V_{\sigma(\lambda)})
\otimes_{\ca{O}_X}\phi^*\bigl(\sheafhom_{\ca{O}_X}
(p^*(V_{\sigma(\lambda)}),\ca{G})^C\bigr)\cong\ca{G}
$$
by the isotypical decomposition formula.  We conclude that the abelian
categories $\lie{M}$ and $\lie{M}^{\tilde{T}}(X)$ are equivalent and,
by virtue of \cite[\S\,2(8)]{quillen;higher-algebraic-k-theory-I},
that the map
\begin{equation}\label{equation;sum}
\bigoplus_{\lambda\in\X(C)}K^T_n(X)\overset\cong\longto
K^{\tilde{T}}_n(X)
\end{equation}
defined by $(a_\lambda)_{\lambda\in C}\mapsto\sum_{\lambda\in
C}e^{\sigma(\lambda)}\cdot\phi^*(a_\lambda)$ is an isomorphism for all
$n\in\N$.  Setting $X=\pt$ and $n=0$ in \eqref{equation;sum} gives
$R(\tilde{T})\cong\bigoplus_{\lambda\in\X(C)}R(T)$, and hence
\begin{equation}\label{equation;sum-point}
R(\tilde{T})\otimes_{R(T)}K^T_n(X)\cong
\bigoplus_{\lambda\in\X(C)}K^T_n(X).
\end{equation}
Combining the isomorphisms \eqref{equation;sum} and
\eqref{equation;sum-point}, we get
$$R(\tilde{T})\otimes_{R(T)}K^T_n(X)\cong K^{\tilde{T}}_n(X).$$
This isomorphism is induced by the functor
$\lie{Rep}(\tilde{T})\times\lie{M}^T(X)\to\lie{M}^{\tilde{T}}(X)$
defined by $(E,\ca{F})\mapsto p^*(E)\otimes_{\ca{O}_X}\phi^*(\ca{F})$,
which proves the first assertion.  Finally, the isomorphism
\eqref{equation;sum} is $R(T)$-linear and the projection of
$K^{\tilde{T}}_n(X)$ onto the direct summand corresponding to
$\lambda=0$ is a left inverse of $\phi^*$.
\end{proof}

A further ingredient we need is a K\"unneth formula of Merkurjev.

\begin{theorem}%
[{\cite[Proposition~4.1]{merkurjev;comparison-equivariant-ordinary}}]%
\label{theorem;merkurjev}
Assume that $\pi_1(G)$ is torsion-free.  For every $G$-variety $X$,
the map
$$R(T)\otimes_{R(G)}K^G_*(X)\to K^T_*(X)$$
defined by $u\otimes b\mapsto u\cdot j^*(b)$ is an isomorphism.
\end{theorem}

A point where the present treatment diverges from that of Section
\ref{section;push-pull} is the absence of a ring structure on the
group $K^G_*(X)$.  Because of this there is no natural sense in which
the $\delta_{\alpha,X}$ are $K^G_*(X)$-linear.  Instead we have the
following statement.

\begin{lemma}\label{lemma;''linear''}
Assume that $\pi_1(G)$ is torsion-free.  Identify the $R(T)$-modules
$K^T_*(X)$ and $R(T)\otimes_{R(G)}K^G_*(X)$ via the K\"unneth
isomorphism.  Then $\delta_X=\partial_{w_0}\otimes1$, where $1$
denotes the identity map of $K^G_*(X)$.
\end{lemma}

\begin{proof}
For all $u\in R(T)$ and $b\in K^G_*(X)$ we have
$$
\delta_X(u\cdot j^*(b))=j^*j_*(u\cdot j^*(b))=j^*(j_*(u)\cdot
b)=j^*j_*(u)\cdot j^*(b)=\partial_{w_0}(u)\cdot j^*(b),
$$
where we used the projection formula and the naturality of $j^*$.
\end{proof}

Under the same assumptions we have
$\delta_{\alpha,X}=\delta_\alpha\otimes1$ for all roots $\alpha$.
With Theorem \ref{theorem;demazure}, Lemma \ref{lemma;snaith-uma},
Theorem \ref{theorem;merkurjev} and Lemma \ref{lemma;''linear''} in
hand, one proves the next assertion in exactly the same way as
Proposition \ref{proposition;lift}.

\begin{proposition}\label{proposition;lift-algebraic}
The operators $\delta_{\alpha,X}$, for $\alpha\in\ca{R}$, together
with the natural $R(T)$-module structure generate a unique
$\ca{D}$-module structure on $K_T^*(X)$.  This $\ca{D}$-module
structure is contravariant with respect to $G$-morphisms of varieties
and isogenies of~$G$.
\end{proposition}

The proof of the following result is now identical to that of Theorem
\ref{theorem;d}.

\begin{theorem}\label{theorem;d-algebraic}
For every $G$-variety $X$, the map $j^*$ is an isomorphism from
$K^G_*(X)$ onto $K^T_*(X)^{I(\ca{D})}$.
\end{theorem}

From this one easily deduces obvious analogues of Theorems
\ref{theorem;weyl} and \ref{theorem;abelian} and Corollary
\ref{corollary;abelian-torsion}%
\eqref{item;torsion-free}--\eqref{item;injective}, which we leave it
to the reader to state.  The analogue of Corollary
\ref{corollary;abelian-torsion}\eqref{item;projective} is as follows.

\begin{corollary}\label{corollary;complete}
Assume that $\k$ is perfect.  Let $X$ be a smooth projective
$G$-variety over $\k$.  Then $j^*\colon K^G_*(X)\to K^T_*(X)^W$ is an
isomorphism.
\end{corollary}

\begin{proof}
It follows from
\cite[Theorem~2]{vezzosi-vistoli;higher-algebraic-diagonalizable} that
the restriction map $K^T_*(X)\to K^T_*(X^T)$ is injective.  Now apply
(the algebraic K-theory analogue of) Corollary
\ref{corollary;abelian-torsion}\eqref{item;injective}.
\end{proof}

For a nonsingular variety the theories $K_G^*(X)$ and $K^G_*(X)$ are
isomorphic and $K^G_*(X)$ is a graded commutative ring.  (See e.g.\
\cite[appendix]{vezzosi-vistoli;higher-algebraic-finite}.)  The result
of Section \ref{section;duality} therefore extends to algebraic
K-theory in the following way.

\begin{proposition}\label{proposition;orthogonal-algebraic}
Assume that $\pi_1(G)$ is torsion-free.  Let $X$ be a smooth
quasi-projective $G$-variety over $\k$.  Then the pairing
$$\ca{P}_X\colon K^T_*(X)\times K^T_*(X)\to K^G_*(X)$$
defined by $\ca{P}_X(a_1,a_2)=j_*(a_1a_2)$ is nonsingular.  Hence
$$
K^T_*(X)\cong\Hom_{K^G_*(X)}(K^T_*(X),K^G_*(X))
$$
as $\Z$-graded left $K^G_*(X)$-modules.
\end{proposition}


\bibliographystyle{amsplain}

\def\cprime{$'$}
\providecommand{\bysame}{\leavevmode\hbox to3em{\hrulefill}\thinspace}
\providecommand{\MR}{\relax\ifhmode\unskip\space\fi MR }
\providecommand{\MRhref}[2]{%
  \href{http://www.ams.org/mathscinet-getitem?mr=#1}{#2}
}
\providecommand{\href}[2]{#2}


\end{document}